\documentclass[11pt]{article}

\usepackage{amssymb,amsmath,amsfonts,amsthm}
\usepackage{latexsym}
\usepackage{graphics}
\usepackage{indentfirst}

\newtheorem*{thm*}{Theorem}
\newtheorem{thm}{Theorem}
\newtheorem{lem}[thm]{Lemma}
\newtheorem{pro}[thm]{Proposition}

\newtheorem{ques}[thm]{Question}

\newcommand{\N}{\mathbb{N}}

\newcommand{\col}{\mathrm{col}}

\begin{document}

\title{A Note on the DP-Chromatic Number of Complete Bipartite Graphs}

\author{Jeffrey A. Mudrock\footnote{Department of Applied Mathematics, Illinois Institute of Technology, Chicago, IL 60616. E-mail: {\tt jmudrock@hawk.iit.edu}} }

\maketitle

\begin{abstract}

DP-coloring (also called correspondence coloring) is a generalization of list coloring recently introduced by  Dvo\v{r}\'{a}k and Postle.  Several known bounds for the list chromatic number of a graph $G$, $\chi_\ell(G)$, also hold for the DP-chromatic number of $G$, $\chi_{DP}(G)$.  On the other hand, there are several properties of the DP-chromatic number that shows that it differs with the list chromatic number.  In this note we show one such property.  It is well known that $\chi_\ell (K_{k,t}) = k+1$ if and only if $t \geq k^k$.  We show that $\chi_{DP} (K_{k,t}) = k+1$ if $t \geq 1 + (k^k/k!)(\log(k!)+1)$, and we show that $\chi_{DP} (K_{k,t}) < k+1$ if $t < k^k/k!$.  

\medskip



\end{abstract}

\section{Introduction}\label{intro}

In this note all graphs are nonempty, finite, simple graphs unless otherwise noted.  Generally speaking we follow West~\cite{W01} for terminology and notation.  For this note the set of natural numbers is $\N =\{1,2,3 \ldots \}$.  The natural log function is denoted $\log$.  Given a set $A$, $\mathcal{P}(A)$ is the power set of $A$.  Also, for any $k \in \N$, $[k] = \{1,2,3, \ldots, k \}$.  If $G$ is a graph and $S, U \subseteq V(G)$, we use $G[S]$ for the subgraph of $G$ induced by $S$, and we use $E_G(S, U)$ for the subset of $E(G)$ with one endpoint in $S$ and one endpoint in $U$.  Also, if $v \in V(G)$ we use $N_G(v)$ for the set of neighbors of $v$ in $G$.

\subsection{List Coloring}

List coloring is a well known variation on the classic vertex coloring problem, and it was introduced independently by Vizing~\cite{V76} and Erd\H{o}s, Rubin, and Taylor~\cite{ET79} in the 1970's.  In the classic vertex coloring problem we wish to color the vertices of a graph $G$ with as few colors as possible so that adjacent vertices receive different colors, a so-called \emph{proper coloring}. The chromatic number of a graph, denoted $\chi(G)$, is the smallest $k$ such that $G$ has a proper coloring that uses $k$ colors.  For list coloring, we associate a \emph{list assignment}, $L$, with a graph $G$ such that each vertex $v \in V(G)$ is assigned a list of colors $L(v)$ (we say $L$ is a list assignment for $G$).  The graph $G$ is \emph{$L$-colorable} if there exists a proper coloring $f$ of $G$ such that $f(v) \in L(v)$ for each $v \in V(G)$ (we refer to $f$ as a \emph{proper $L$-coloring} of $G$).  A list assignment $L$ is called a \emph{k-assignment} for $G$ if $|L(v)|=k$ for each $v \in V(G)$.  The \emph{list chromatic number} of a graph $G$, denoted $\chi_\ell(G)$, is the smallest $k$ such that $G$ is $L$-colorable whenever $L$ is a $k$-assignment for $G$.  We say $G$ is \emph{$k$-choosable} if $k \geq \chi_\ell(G)$.

It is immediately obvious that for any graph $G$, $\chi(G) \leq \chi_\ell(G)$.  Erd\H{o}s, Rubin, and Taylor~\cite{ET79} studied the equitable choosability of $K_{m,m}$ and observed that if $m = \binom{2k-1}{k}$, then $\chi_\ell(K_{m,m}) > k$.  The following related result is often attributed to Vizing~\cite{V76} or Erd\H{o}s, Rubin, and Taylor~\cite{ET79}, but it is best described as a folklore result. 

\begin{thm} \label{thm: listbipartite}
For $k \in \N$, $\chi_\ell(K_{k,t})=k+1$ if and only if $t \geq k^k$.
\end{thm}

\noindent We study the analogue of Theorem~\ref{thm: listbipartite} for DP-coloring.

\subsection{DP-coloring}

Dvo\v{r}\'{a}k and Postle~\cite{DP15} introduced DP-coloring (they called it correspondence coloring) in 2015 in order to prove that every planar graph without cycles of lengths 4 to 8 is 3-choosable.  Intuitively, DP-coloring is a generalization of list coloring where each vertex in the graph still gets a list of colors but identification of which colors are different can vary from edge to edge.  Following~\cite{BK17}, we now give the formal definition.  Suppose $G$ is a graph.  A \emph{cover} of $G$ is a pair $\mathcal{H} = (L,H)$ consisting of a graph $H$ and a function $L: V(G) \rightarrow \mathcal{P}(V(H))$ satisfying the following four requirements:

\vspace{5mm}

\noindent(1) the sets $\{L(u) : u \in V(G) \}$ form a partition of $V(H)$; \\
(2) for every $u \in V(G)$, the graph $H[L(u)]$ is complete; \\
(3) if $E_H(L(u),L(v))$ is nonempty, then $u=v$ or $uv \in E(G)$; \\
(4) if $uv \in E(G)$, then $E_H(L(u),L(v))$ is a matching (the matching may be empty). 

\vspace{5mm}

\noindent Suppose $\mathcal{H} = (L,H)$ is a cover of $G$.  We say $\mathcal{H}$ is \emph{$k$-fold} if $|L(u)|=k$ for each $u \in V(G)$.  An $\mathcal{H}$-coloring of $G$ is an independent set in $H$ of size $|V(G)|$.  It is immediately clear that $I \subseteq V(G)$ is an $\mathcal{H}$-coloring if and only if $|I \cap L(u)|=1$ for each $u \in V(G)$.

The \emph{DP-chromatic number} of a graph $G$, $\chi_{DP}(G)$, is the smallest $k \in \N$ such that $G$ admits an $\mathcal{H}$-coloring for every $k$-fold cover $\mathcal{H}$ of $G$.  Suppose we wish to prove $\chi_{DP}(G) \leq k$.  Since every $k$-fold cover of $G$ is isomorphic to a subgraph of some $k$-fold cover, $\mathcal{H}' = (L',H')$, of $G$ with the property that $E_{H'}(L'(u),L'(v))$ is a perfect matching whenever $uv \in E(G)$, we need only show that $G$ has an $\mathcal{H}$-coloring whenever $\mathcal{H} = (L,H)$ is a $k$-fold cover of $G$ such that $E_{H}(L(u),L(v))$ is a perfect matching for each $uv \in E(G)$.

Given a list assignment, $L$, for a graph $G$, it is easy to construct a cover $\mathcal{H}$ of $G$ such that $G$ has an $\mathcal{H}$-coloring if and only if $G$ has a proper $L$-coloring (see~\cite{BK17}).  It follows that $\chi_\ell(G) \leq \chi_{DP}(G)$.  This inequality may be strict since it is easy to prove that $\chi_{DP}(C_n) = 3$ whenever $n \geq 3$, but the list chromatic number of any even cycle is 2 (see~\cite{BK17} and~\cite{ET79}).

We now briefly discuss some similarities between the DP-coloring and list coloring.  First, notice that like $k$-choosability, the graph property of having DP-chromatic number at most $k$ is monotone.  It is also clear that, as in the context of list coloring, if $\chi_{DP}(G) = k$, then an $\mathcal{H}$-coloring of $G$ exists whenever $\mathcal{H}$ is an $m$-fold cover of $G$ with $m \geq k$.  The \emph{coloring number} of a graph $G$, denoted $\col(G)$, is the smallest integer $d$ for which there exists an ordering, $v_1, v_2, \ldots, v_n$, of the elements in $V(G)$ such that each vertex $v_i$ has at most $d-1$ neighbors among $v_1, v_2, \ldots, v_{i-1}$.  It is easy to prove that $\chi_\ell(G) \leq \chi_{DP}(G) \leq \col(G)$.  Thomassen~\cite{T94} famously proved that every planar graph is 5-choosable, and Dvo\v{r}\'{a}k and L. Postle~\cite{DP15} observed that the DP-chromatic number of every planar graph is at most 5.  Also, Molloy~\cite{M17} recently improved a theorem of Johansson, and showed that every triangle-free graph $G$ with maximum degree $\Delta(G)$ satisfies $\chi_\ell(G) \leq (1 + o(1)) \Delta(G)/ \log(\Delta(G))$.  Bernshteyn~\cite{B17} subsequently showed that this bound also holds for the DP-chromatic number. 

On the other hand, Bernshteyn~\cite{B16} showed that if the average degree of a graph $G$ is $d$, then $\chi_{DP}(G) = \Omega(d/ \log(d))$.  This is in stark contrast to the celebrated result of Alon~\cite{A00} which says $\chi_\ell(G) = \Omega(\log(d))$.  It was also recently shown in~\cite{BK17} that there exist planar bipartite graphs with DP-chromatic number 4 even though the list chromatic number of any planar bipartite graph is at most 3~\cite{AT92}.  A famous result of Galvin~\cite{G95} says that if $G$ is a bipartite multigraph and $L(G)$ is the line graph of $G$, then $\chi_\ell(L(G)) = \chi(L(G)) = \Delta(G)$.  However, it is also shown in~\cite{BK17} that every $d$-regular graph $G$ satisfies $\chi_{DP}(L(G)) \geq d+1$.            

\subsection{Outline of Results and an Open Question}

In this note we present some results on the DP-chromatic number of complete bipartite graphs.  By what was mentioned in the previous subsection, we know that if $k,t \in \N$, $\chi_{DP}(K_{k,t}) \leq \col(K_{k,t}) \leq k+1$.  For the remainder of this note, for each $k \in \N$, let $\mu(k)$ be the smallest natural number $l$ such that $\chi_{DP}(K_{k,l})=k+1$.  We have that $\mu(k)$ exists for each $k \in \N$ since we know by Theorem~\ref{thm: listbipartite}, 
$$k+1 = \chi_\ell(K_{k,k^k}) \leq \chi_{DP}(K_{k,k^k}) \leq k+1.$$  
\noindent This means that $\mu(k) \leq k^k$ for each $k \in \N$.  The following proposition is also clear.

\begin{pro} \label{pro: obvious}
For $k \in \N$, $\chi_{DP}(K_{k,t})=k+1$ if and only if $t \geq \mu(k)$
\end{pro}

\begin{proof}
If $t \geq \mu(k)$, $k+1 = \chi_{DP} (K_{k, \mu(k)}) \leq \chi_{DP} (K_{k, t}) \leq k+1$ since $K_{k,\mu(k)}$ is a subgraph of $K_{k,t}$.  Conversely, if $\chi_{DP}(K_{k,t})=k+1$, then $\mu(k) \leq t$ by the definition of $\mu(k)$. 
\end{proof}

Computing $\mu(k)$ is easy when $k=1,2$.  Clearly, $\mu(1)=1$.  Also, $\mu(2)=2$ follows from the fact that $\chi_{DP}(K_{2,1}) \leq \col(K_{2,1})=2$, and the fact that $K_{2,2}$ is a 4-cycle which implies $\chi_{DP}(K_{2,2})=3$.  We have a tedious argument that shows $\mu(3)=6$, which for the sake of brevity, we do not present in this note.  The following question lead to the discovery of both results in this note.

\begin{ques} For each $k \geq 4$, what is the exact value of $\mu(k)$?
\end{ques}

We obtain an upper bound and lower bound on $\mu(k)$.  Our first result gives us a lower bound.

\begin{thm} \label{thm: lowerbound}
For $k \in \N$, if $t < \frac{k^k}{k!}$, then $\chi_{DP}(K_{k,t}) < k+1$.
\end{thm}

Theorem~\ref{thm: lowerbound} tells us that $\lceil k^k/k! \rceil \leq \mu(k)$ notice this lower bound is tight for $k=1,2$, and it is 1 away from being tight for $k=3$.  We then use a simple probabilistic argument to prove our second result which gives us an upper bound on $\mu(k)$. 

\begin{thm} \label{thm: upperbound} For $t \in \N$ let 
$$ m = t+ \left \lfloor k^k \left ( 1- \frac{k!}{k^k} \right )^t \right \rfloor$$
Then, $\chi_{DP}(K_{k,m}) = k+1$.
\end{thm} 

Theorems~\ref{thm: lowerbound} and~\ref{thm: upperbound} imply 
$$\left \lceil \frac{k^k}{k!} \right \rceil \leq \mu(k) \leq 1 + \frac{k^k(\log(k!)+1)}{k!}.$$ 
\noindent We suspect that $\lceil k^k/k! \rceil$ is closer to the exact value of $\mu(k)$ than the upper bound.

\section{Proofs of Results}

In this section we prove Theorems~\ref{thm: lowerbound} and~\ref{thm: upperbound}.  We begin with a definition.  Suppose that $\mathcal{H} = (L, H)$ is a $k$-fold cover of $G$.  For any $v \in V(G)$, we say an independent set, $I$, in $H[ \cup_{u \in N_G(v)} L(u)]$ is \emph{bad for $v$} if $|I| = |N_G(v)|$ and for each $w \in L(v)$, $w$ is adjacent to some vertex in $I$.  Notice that if $I$ is bad for $v$, then an $\mathcal{H}$-coloring of $G$ cannot contain $I$.

In this section we often have $G = K_{k,t}$, and we always suppose $G$ has bipartition $X= \{v_1, v_2, \ldots, v_{k} \}$, $Y= \{u_1, u_2, \ldots, u_{t}\}$.  We now mention an idea used frequently in this section.  Notice that if $\mathcal{H} = (L, H)$ is a $k$-fold cover $G = K_{k,t}$, then there are precisely $k^k$ independent sets of size $k$ in $H[ \cup_{v \in X} L(v)]$.  If all of these independent sets are bad for at least one vertex in $Y$, then there is no $\mathcal{H}$-coloring of $G$.  We now prove a lemma which gives us a bound on how many independent sets of size $k$ in $H[ \cup_{v \in X} L(v)]$ can be bad for a vertex in $Y$.

\begin{lem} \label{lem: badI}
Suppose $G$ is a graph, $v \in V(G)$, and $|N_G(v)| = k$.  Suppose that $\mathcal{H} = (L, H)$ is a $k$-fold cover of $G$.  Then, there are at most $k!$ distinct independent sets in $H[ \cup_{u \in N_G(v)} L(u)]$ that are bad for $v$.
\end{lem} 

\begin{proof}
The result is obvious when $k=1$.  So, suppose $k \geq 2$. We let $H' = H[ \cup_{u \in N_G(v)} L(u)]$.  Suppose that $N_G(v) = \{v_1, v_2, \ldots, v_{k} \}$.   

\par

Let $\mathcal{C}$ denote the set of bijective functions from $[k]$ to $L(v)$.  Let $\mathcal{I}$ denote the set of all independent sets in $H'$ that are bad for $v$.  We are done if $\mathcal{I} = \emptyset$, so we assume $\mathcal{I} \neq \emptyset$.  We now define an injective mapping, $f: \mathcal{I} \rightarrow \mathcal{C}$.  For $I \in \mathcal{I}$ suppose that $I = \{u_1, u_2, \ldots, u_k \}$ where $u_i \in L(v_i)$ (we know that $|I \cap L(v_i)|=1$ for each $i \in [k]$).  Suppose that for each $i \in [k]$, $w_i$ is the one vertex in $L(v)$ to which $u_i$ is adjacent.  Then, let $\sigma_I: [k] \rightarrow L(v)$ be the function defined by $\sigma_I(i) = w_i$.  Since $I$ is bad for $v$, we know that $\sigma_I \in \mathcal{C}$.  So, we can let $f(I) = \sigma_I$

\par

To see that $f$ is injective, suppose that $I = \{u_1, u_2, \ldots, u_k \}$ and $I' = \{u'_1, u'_2, \ldots, u'_k \}$ are distinct elements of $\mathcal{I}$ where $u_i, u'_i \in L(v_i)$ for each $i \in [k]$.  This means that there must be a $j \in [k]$ such that $u_j \neq u'_j$.  Since $E_H(L(v), L(v_j))$ is a matching, we know that $u_j$ and $u'_j$ are adjacent to distinct vertices in $L(v)$.  Thus, $f(I) \neq f(I')$.  The fact that $f$ is injective immediately implies that $|\mathcal{I}| \leq |\mathcal{C}|=k!$.   

\end{proof} 

We are now ready to prove Theorem~\ref{thm: lowerbound}.

\begin{proof}
We suppose $k \geq 3$ since the result is clear for $k=1,2$.  We also assume $t \in \N$ since the result is clear when $t=0$.  Suppose $G = K_{k,t}$.

\par

Let $\mathcal{H} = (L, H)$ be an arbitrary $k$-fold cover of $G$.  Let $H' = H[ \cup_{i=1}^{k} L(v_i)]$.  It is clear that there are $k^k$ independent sets of size $k$ in $H'$.  Moreover, we know from Lemma~\ref{lem: badI} that there are at most $k!$ independent sets in $H'$ that are bad for $u_j$ for each $j \in [t]$.  Since 
$$k^k - t(k!) > 0,$$
there is an independent set, $I$ in $H'$ such that $|I|=k$ and $I$ is not bad for any vertex in $Y$.  Thus, for each $j \in [t]$, we can find a $w_j \in L(u_j)$ that is not adjacent to any vertex in $I$.  Finally, $I \cup \{w_1, w_2, \ldots, w_{t} \}$ is an $\mathcal{H}$-coloring of $G$.
\end{proof}   

We now prove Theorem~\ref{thm: upperbound}

\begin{proof}
We suppose $k \geq 2$ since the result is clear for $k=1$.  Suppose $G = K_{k,t}$.  We form a $k$-fold cover of $G$ by the following (partially random) process. We begin by letting $L(v_i) = \{(v_i, l) : l \in [k] \}$ and $L(u_j) = \{(u_j, l) : l \in [k] \}$ for each $i \in [k]$ and $j \in [t]$.  Let graph $H$ have vertex set 
$$\left (\bigcup_{i=1}^{k} L(v_i) \right ) \bigcup \left ( \bigcup_{j=1}^{t} L(u_j) \right ).$$  
\noindent Also, draw edges in $H$ so that $H[L(v)]$ is a clique for each $v \in V(G)$.  Finally, for each $i \in [k]$ and $j \in [t]$, uniformly and randomly choose a perfect matching between $L(v_i)$ and $L(u_j)$ from the $k!$ possible perfect matchings.  It is easy to see that $\mathcal{H}=(L,H)$ is a $k$-fold cover of $G$.

\par

Note that there are exactly $k^k$ independent sets of size $k$ in $H[ \cup_{i=1}^{k} L(v_i) ]$.  Suppose we name the $k^k$ functions from $[k]$ to $[k]$: $f_1, f_2, f_3, \ldots, f_{k^k}$.  Then the $k^k$ independent sets of size $k$ in $H[ \cup_{i=1}^{k} L(v_i) ]$ are precisely: $I_1, I_2, \ldots, I_{k^k}$ where $I_i = \{(u_l, f_i(l)) : l \in [k] \}$.

\par

Suppose that for each $i \in [k^k]$, $E_i$ is the event that $I_i$ is not bad for any vertex in $Y$.  For any vertex $u \in Y$, it is easy to see that the probability that $I_i$ is bad for $u$ is 
$$\frac{k!((k-1)!)^k}{(k!)^k} = \frac{k!}{k^k}.$$  
Thus, $P[E_i] = (1-k!/k^k)^t$.  Let $X_i$ be the random variable that is 1 when $E_i$ occurs, and it is 0 otherwise.  Let $X = \sum_{i=1}^{k^k} X_i$.  By linearity of expectation,
$$E[ X ] = k^k \left ( 1- \frac{k!}{k^k} \right )^t. $$
Let $z = \lfloor E[X] \rfloor$.  We can find a $k$-fold cover, $\mathcal{H}' = (L',H')$, of $G$ such that at most $z$ of the independent sets of size $k$ in $H'[ \cup_{i=1}^{k} L'(v_i) ]$ are not bad for any vertex in $Y$.  Suppose we call such independent sets: $I_{a_1}, I_{a_2}, \ldots, I_{a_{r}}$ (we know $r \leq z$).  

\par

Starting with $G$, we create a copy of $K_{k,t+r}$, called $M$, by adding $r$ new vertices, $w_1, \ldots, w_{r}$, to $Y$.  We construct a $k$-fold cover of $M$ starting from $\mathcal{H}'$ as follows.  With each $w_i$ we associate $k$ vertices, $L''(w_i)$, and we add these vertices to $H'$ along with edges so that the vertices in $L''(w_i)$ are pairwise adjacent.  Then, for $i \in [r]$, we create a matching between $L'(v_j)$ and $L''(w_i)$ for each $j \in [k]$ so that $I_{a_i}$ is bad for $w_i$.  The result is a $k$-fold cover, $\mathcal{H}''$, of $M$ with the property that there is no $\mathcal{H}''$-coloring of $M$.  Thus, $k+1 = \chi_{DP}(K_{k,t+r}) \leq \chi_{DP}(K_{k,m})$.        
\end{proof}

Letting $t = \lceil k^k \log(k!)/k! \rceil$, we note
$$ t+ \left \lfloor k^k \left ( 1- \frac{k!}{k^k} \right )^t \right \rfloor \leq \left \lceil \frac{k^k \log(k!)}{k!} \right \rceil + \frac{k^k}{k!} \leq 1 + \frac{k^k(\log(k!)+1)}{k!}$$
and the upper bound on $\mu(k)$ mentioned in the previous section follows.

\vspace{5mm}

\noindent \textbf{Acknoledgement:}  The author would like to thank Anton Bernshteyn, Hemanshu Kaul, and Alexandr Kostochka for their helpful comments on this note.

\end{document}